\documentclass[12pt]{article}
\usepackage{amsfonts,amssymb,amsmath,amsthm}
\usepackage[english]{babel}
\date{}
\newcommand{\R}{{\mathbb R}}
\newcommand{\Z}{{\mathbb Z}}
\newcommand{\N}{{\mathbb N}}
\newcommand{\T}{{\mathbb T}}
\newcommand{\D}{{\mathcal D}}
\newcommand{\Cl}{\mathrm{Cl}\,}
\newcommand{\LL}{{\mathcal L}}
\newcommand{\const}{{\rm const}}
\newcommand{\essinf}{\mathop{\rm ess\,inf}}
\newcommand{\esssup}{\mathop{\rm ess\,sup}}
\newcommand{\esslim}{\mathop{\rm ess\,lim}}
\newcommand{\sign}{\mathop{\rm sign}}
\renewcommand{\div}{{\rm div}}
\newtheorem{theorem}{Theorem}
\newtheorem{lemma}{Lemma}

\newtheorem{corollary}{Corollary}

\theoremstyle{definition}
\newtheorem{definition}{Definition}
\newtheorem{remark}{Remark}

\def\Xint#1{\mathchoice
{\XXint\displaystyle\textstyle{#1}}%
{\XXint\textstyle\scriptstyle{#1}}%
{\XXint\scriptstyle\scriptscriptstyle{#1}}%
{\XXint\scriptscriptstyle\scriptscriptstyle{#1}}%
\!\int}
\def\XXint#1#2#3{{\setbox0=\hbox{$#1{#2#3}{\int}$ }
\vcenter{\hbox{$#2#3$ }}\kern-.57\wd0}}

\def\dashint{\Xint-}
\numberwithin{equation}{section}

\title{On almost periodic viscosity solutions to Hamilton-Jacobi equations}
\author{Evgeny Yu. Panov}
\begin{document}
\maketitle
\begin{abstract}
We establish that a viscosity solution to a multidimensional Hamilton-Jacobi equation with Bohr almost periodic initial data remains to be spatially almost periodic and the additive subgroup generated by its spectrum does not increase in time. In the case of one space variable and a non-degenerate hamiltonian we prove the decay property of almost periodic viscosity solutions when time $t\to+\infty$. For periodic solutions the more general result is proved on unconditional asymptotic convergence of a viscosity solution to a traveling wave.
\end{abstract}

\section{Introduction}
In the half-space $\Pi=\R_+\times\R^n$, $\R_+=(0,+\infty)$, we consider the Cauchy problem for a
first order Hamilton-Jacobi equation
\begin{equation}\label{1}
u_t+H(\nabla_x u)=0
\end{equation}
with merely continuous hamiltonian function $H(p)\in C(\R^n)$, and with initial condition
\begin{equation}\label{2}
u(0,x)=u_0(x)\in BUC(\R^n),
\end{equation}
where $BUC(\R^n)$ denotes the Banach space of bounded uniformly continuous functions on $\R^n$
equipped with the uniform norm $\|u\|_\infty=\esssup|u(x)|$. We recall the notions of
superdifferential $D^+u$ and subdifferential $D^-u$ of a continuous function $u(t,x)\in C(\Pi)$:
\begin{eqnarray*}
D^+u(t_0,x_0)=\{ \ \nabla \varphi(t_0,x_0) \ | \ \varphi(t,x)\in C^1(\Pi), \\ (t_0,x_0) \mbox{ is a point of local maximum of } u-\varphi \ \}, \\
D^-u(t_0,x_0)=\{ \ \nabla \varphi(t_0,x_0) \ | \ \varphi(t,x)\in C^1(\Pi), \\ (t_0,x_0) \mbox{ is a point of local minimum of } u-\varphi \ \}.
\end{eqnarray*}
Let us denote by $BUC_{loc}(\bar\Pi)$ the space of continuous functions on $\bar\Pi=\Cl\Pi=[0,+\infty)\times\R^n$, which are bounded and uniformly continuous in any layer $[0,T)\times\R^n$, $T>0$.

\begin{definition}[see \cite{CL,CEL}]\label{def1}

A function $u(t,x)\in BUC_{loc}(\bar\Pi)$ is called a viscosity subsolution (v.subs. for short) of problem (\ref{1}), (\ref{2}) if
$u(0,x)\le u_0(x)$ and $s+H(p)\le 0$ for all $(s,p)\in D^+u(t,x)$, $(t,x)\in\Pi$.

A function  $u(t,x)\in BUC_{loc}(\bar\Pi)$ is called a viscosity supersolution (v.supers.) of problem (\ref{1}), (\ref{2}) if
$u(0,x)\ge u_0(x)$ and $s+H(p)\ge 0$ for all $(s,p)\in D^-u(t,x)$, $(t,x)\in\Pi$.

Finally, $u(t,x)\in BUC_{loc}(\bar\Pi)$ is called a viscosity solution (v.s.) of (\ref{1}), (\ref{2}) if it is a v.subs. and a v.supers.
of this problem simultaneously.
\end{definition}
The theory of viscosity solutions was developed in \cite{CL,CEL} for general multidimensional Hamilton-Jacobi equation,
this theory extended the earlier results of S.N. Kruzhkov \cite{KrHJ,KrHJ1}.

It is known that for each $u_0(x)\in BUC(\R^n)$ there exists a unique v.s. of problem (\ref{1}), (\ref{2}). The uniqueness readily follows from the more general comparison principle.

\begin{theorem}\label{th1}
Let $u_1(t,x), u_2(t,x)\in  BUC_{loc}(\bar\Pi)$ be a v.subs. and a v.supers. of (\ref{1}), (\ref{2}) with initial data $u_{10}(x), u_{20}(x)$, respectively. Assume that $u_{10}(x)\le u_{20}(x)$ $\forall x\in\R^n$.
Then $u_1(t,x)\le u_2(t,x)$ $\forall (t,x)\in\Pi$.
\end{theorem}

For the proof of Theorem~\ref{th1}, we refer to \cite{CEL}.

\begin{corollary}\label{cor1}
Let $u_1(t,x), u_2(t,x)\in  BUC_{loc}(\bar\Pi)$ be v.s. of (\ref{1}), (\ref{2}) with initial data $u_{10}(x), u_{20}(x)$, respectively. Then for all $t>0$
$$
\essinf (u_{10}(x)-u_{20}(x))\le u_1(t,x)-u_2(t,x)\le \esssup (u_{10}(x)-u_{20}(x)).
$$
In particular, $\|u_1-u_2\|_\infty\le\|u_{10}-u_{20}\|_\infty$.

\end{corollary}

\begin{proof}
We denote
$$
a=\essinf (u_{10}(x)-u_{20}(x)), \quad b=\esssup (u_{10}(x)-u_{20}(x)),
$$
and observe that the functions $a+u_2(t,x)$, $b+u_2(t,x)$ a v.s. of (\ref{1}), (\ref{2}) with initial data $a+u_{20}(x)$, $b+u_{20}(x)$, respectively. Since $a+u_{20}(x)\le u_{10}(x)\le b+u_{20}(x)$, then by Theorem~\ref{1}
$a+u_2(t,x)\le u_1(t,x)\le b+u_2(t,x)$ $\forall (t,x)\in\Pi$, which completes the proof.
\end{proof}

\begin{lemma}\label{lem1}
Let $H(p,q)\in C(\R^n\times\R^m)$. We consider the equation
\begin{equation}\label{a1}
u_t+H(\nabla_x u,\nabla_y u)=0
\end{equation}
in the half-space $\{ \ (t,x,y) \ | \ t>0, x\in\R^n, y\in\R^m \ \}$. Then $u(t,y)$ is a non-depending on $x$ v.s. of (\ref{a1}) if and only if $u(t,y)$ is a v.s. of the reduced equation
$$
u_t+H(0,\nabla_y u)=0, \quad (t,y)\in\in\R_+\times\R^m.
$$
\end{lemma}

\begin{proof}
The assertion of lemma readily follows from the evident equalities
$$D^\pm u(t,x,y)=\{ (s,0,q)\in \R\times\R^n\times\R^m \ | \ (s,q)\in D^\pm u(t,y) \ \}.$$
\end{proof}

\begin{lemma}\label{lem2}
Let $y=Ax$, $\displaystyle y_i=\sum_{j=1}^n a_{ij}x_j$, $i=1,\ldots,n$, be a non-degenerate linear operator on $\R^n$, $v_0(y)\in BUC(\R^n)$, $v(t,y)\in BUC_{loc}(\Pi)$. Then the function $u(t,x)=v(t,Ax)$ is a v.s. of (\ref{1}) with initial data $u_0(x)=v_0(Ax)$ if and only if $v(t,y)$ is a v.s. of the problem
$$
v_t+H(A^*\nabla_y v)=0, \quad v(0,y)=v_0(y),
$$
where $A^*$ is a conjugate operator (so that $\displaystyle (A^*\nabla_y v)_j=\sum_{i=1}^n a_{ij}\partial_{y_i} v$, $j=1,\ldots,n$~).
\end{lemma}

\begin{proof}
The statement of Lemma~\ref{lem2} follows from the fact that $(t_0,x_0)$ is a point of local maximum (minimum) of $u(t,x)-\psi(t,Ax)$, with $\psi(t,y)\in C^1(\Pi)$,
if and only if $(t_0,Ax_0)$ is a point of local maximum (minimum) of $v(t,y)-\psi(t,y)$ and from the classical identity $A^*\nabla_y \psi(t,y)=\nabla_x \psi(t,Ax)$, $y=Ax$.
\end{proof}

In the half space $\R_+\times\R^n\times\R^m$ we consider the Cauchy problem
for equation
\begin{equation}\label{l1}
u_t+H(\nabla_x u)=0, \quad u=u(t,x,y), \ (t,x,y)\in \R_+\times\R^n\times\R^m,
\end{equation}
with initial condition
\begin{equation}\label{l2}
u(0,x,y)=u_0(x,y)\in BUC(\R^n\times\R^m).
\end{equation}
\begin{lemma}\label{lem3}
A function $u(t,x,y)\in BUC_{loc}(\R_+\times\R^n\times\R^m)$ is a v.s. of (\ref{l1}),
(\ref{l2}) if and only if for all fixed $y\in\R^m$ the functions $u^y(t,x)=u(t,x,y)$
is a v.s. of (\ref{1}), (\ref{2}) with initial data $u_0^y(x)=u_0(x,y)$.
\end{lemma}

\begin{proof}
Let $u(t,x,y)$ be a v.s. of (\ref{l1}), (\ref{l2}), and $y_0\in\R^m$. We assume that $\varphi(t,x)\in C^1(\Pi)$ and
$(t_0,x_0)\in\Pi$ is a point of local maximum of $u^{y_0}-\varphi$. Moreover, replacing $\varphi$ by $\varphi(t,x)+(t-t_0)^2+|x-x_0|^2+u(t_0,x_0,y_0)-\varphi(t_0,x_0)$ (here and in the sequel we denote by $|z|$ the Euclidean norm of finite-dimensional vector $z$~), we can suppose, without loss of generality, that $(t_0,x_0)\in\Pi$ is a point of strict local maximum of $u^{y_0}-\varphi$, and that in this point
$u^{y_0}(t_0,x_0)-\varphi(t_0,x_0)=0$.  Therefore, there exists $c>0$ such that
$$
\varphi(t,x)-u(t,x,y_0)>c \quad \forall (t,x)\in\Pi, \ (t-t_0)^2+|x-x_0|^2=r^2,
$$
where $r\in (0,t_0)$ is sufficiently small.
By the continuity there exists $h>0$ such that $\varphi(t,x)-u(t,x,y)>c/2$ for all $(t,x,y)\in\R_+\times\R^n\times\R^m$, $(t-t_0)^2+|x-x_0|^2=r^2$, $|y-y_0|\le h$. We can choose such $C_0>0$ that
$$
C_0h^2-c>\max\{ \ u(t,x,y)-\varphi(t,x) \ | \ (t-t_0)^2+|x-x_0|^2\le r^2, \ |y-y_0|=h \ \}.
$$
Then for each natural $k>C_0$ the function $p_k(t,x,y)=\varphi(t,x)+k|y-y_0|^2\in C^1(\R_+\times\R^n\times\R^m)$ and satisfies the property
\begin{equation}\label{l3}
p_k(t,x,y)-u(t,x,y)>c/2>0=p_k(y_0,x_0,y_0)-u(t_0,x_0,y_0)
\end{equation}
$\forall (t,x,y)\in \partial V_{r,h}$, where we denote by $V_{rh}$ the domain
$$
V_{rh}=\{ \ (t,x,y)\in \R_+\times\R^n\times\R^m \ | \ (t-t_0)^2+|x-x_0|^2<r^2, \ |y-y_0|<h \ \}.
$$
In view of (\ref{l3}) the point $(t_k,x_k,y_k)$ such that
$$
p_k(t_k,x_k,y_k)-u(t_k,x_k,y_k)=\min\limits_{(t,x,y)\in \Cl V_{rh}} (p_k(t,x,y)-u(t,x,y))
$$
lies in $V_{rh}$ and, therefore, it is a point of local maximum of the difference $u(t,x,y)-p_k(t,x,y)$.
Since $\nabla p_k(t,x,y)=(\partial_t\varphi(t,x),\nabla_x\varphi(t,x),2k(y-y_0))$, then by the definition of v.s. of
(\ref{l1})
\begin{equation}\label{l4}
\partial_t\varphi(t_k,x_k)+H(\nabla_x\varphi(t_k,x_k))\le 0.
\end{equation}
Since $\min\limits_{(t,x,y)\in \Cl V_{rh}} (p_k(t,x,y)-u(t,x,y))\le p_k(t_0,x_0,y_0)-u(t_0,x_0,y_0)=0$, then  $k|y_k-y_0|^2\le m=\max\limits_{(t,x,y)\in \Cl V_{rh}} (u(t,x,y)-\varphi(t,x))$. In particular $y_k\to y_0$ as $k\to\infty$. Taking into account that $(t_0,x_0)$ is a point of strict local minimum of $u(t,x,y_0)-\varphi(t,x)$,
we derive that $(t_k,x_k)\to (t_0,x_0)$ as $k\to\infty$. Therefore, it follows from (\ref{l4}) in the limit as $k\to\infty$ that
$$
\partial_t\varphi(t_0,x_0)+H(\nabla_x\varphi(t_0,x_0))\le 0.
$$
This means that $u(t,x,y_0)$ is a v.subs. of (\ref{1}).
By the similar reasons we obtain that
$$
\partial_t\varphi(t_0,x_0)+H(\nabla_x\varphi(t_0,x_0))\ge 0
$$
whenever $(t_0,x_0)$ is a point of strict local minimum of $u(t,x,y_0)-\varphi(t,x)$, where $\varphi(t,x)\in C^1(\Pi)$, that is $u(t,x,y_0)$ is a v.supers. of (\ref{1}). Thus, $u(t,x,y_0)$ is a v.s. of (\ref{1}) for each $y_0\in\R^m$.

Conversely, assume that $u^y(t,x)$ is a v.s. of (\ref{1}) for every $y\in\R^m$. Suppose $\varphi(t,x,y)\in C^1(\R_+\times\R^n\times\R^m)$ and that $(t_0,x_0,y_0)$ is a point of local maximum (minimum) of $u(t,x,y)-\varphi(t,x,y)$. Then the point $(t_0,x_0)\in\Pi$ is a point of local maximum (minimum) of
$u^{y_0}(t,x)-\varphi(t,x,y_0)$. Since $u^{y_0}$ is a v.s. of (\ref{1}) then
$\varphi_t(t_0,x_0,y_0)+H(\nabla_x\varphi(t_0,x_0,y_0))\le 0$ (~respectively, $\varphi_t(t_0,x_0,y_0)+H(\nabla_x\varphi(t_0,x_0,y_0))\ge 0$~). Hence, $u(t,x,y)$ is a v.s. of (\ref{l1}).
To complete the proof it only remains to notice that initial condition (\ref{l2}) is satisfied if and only if $u^y(t,x)$ satisfies (\ref{2}) with initial data $u_0^y$ for all $y\in\R^m$.
\end{proof}

\section{Almost periodic viscosity solutions}

Recall that the space $AP(\R^n)$ of Bohr (or uniform) almost periodic functions is a closure of trigonometric polynomials,
i.e. finite sums $\sum a_\lambda e^{2\pi i\lambda\cdot x}$, in the space $BUC(\R^n)$ (by $\cdot$ we denote the inner product in $\R^n$). It is clear that $AP(\R^n)$ contains continuous periodic functions (with arbitrary lattice of periods). Originally, almost periodic functions were defined with the help of notion of almost-periods, see \cite{LevZh} for details. Let $C_R$ be the cube
$$\{ \ x=(x_1,\ldots,x_n)\in\R^n \ | \ |x|_\infty=\max_{i=1,\ldots,n}|x_i|\le R/2 \ \}, \quad R>0.$$
It is known (see for instance \cite{LevZh}~) that for each function $u\in AP(\R^n)$ there exists the mean value
$$\bar u=\dashint_{\R^n} u(x)dx\doteq\lim\limits_{R\to+\infty}R^{-n}\int_{C_R} u(x)dx$$ and, more generally, the Bohr-Fourier coefficients
$$
a_\lambda=\dashint_{\R^n} u(x)e^{-2\pi i\lambda\cdot x}dx, \quad\lambda\in\R^n.
$$
The set
$$ Sp(u)=\{ \ \lambda\in\R^n \ | \ a_\lambda\not=0 \ \} $$ is called the spectrum of an almost periodic function  $u(x)$. It is known \cite{LevZh}, that the spectrum $Sp(u)$ is at most countable.

Now we assume that the initial function $u_0(x)\in AP(\R^n)$. Let $M_0$ be the smallest additive subgroup of $\R^n$ containing $Sp(u_0)$. Notice that in the case when $u_0$ is a continuous periodic function $M_0$ coincides with the dual lattice to the lattice of periods.

Our first result is the following.

\begin{theorem}\label{th2}
Let $u(t,x)$ be a unique v.s. of (\ref{1}), (\ref{2}). Then $u(t,\cdot)\in C([0,+\infty),AP(\R^n))$ and for all $t>0$ $Sp(u(t,\cdot))\subset M_0$.
\end{theorem}

\begin{proof}
We first assume that the initial function is a trigonometric polynomial $\displaystyle u_0(x)=\sum_{\lambda\in S}a_\lambda e^{2\pi i\lambda\cdot x}$. Here $S=Sp(u_0)\subset\R^n$ is a finite set.  Then the subgroup $M_0$ is a finite generated torsion-free abelian group and therefore it is a free abelian group of
finite rank (see \cite{Lang}). Hence, there is a basis $\lambda_j\in M_0$, $j=1,\ldots,m$, so that every element $\lambda\in M_0$ can be uniquely represented as $\displaystyle\lambda=\lambda(\bar k)=\sum_{j=1}^m k_j\lambda_j$, $\bar k=(k_1,\ldots,k_m)\in\Z^m$. In particular, we can represent the initial function as
$$u_0(x)=\sum_{\bar k\in J} a_{\bar k}e^{2\pi i\sum_{j=1}^m k_j\lambda_j\cdot x}, \quad a_{\bar k}\doteq a_{\lambda(\bar k)},$$
where $J=\{ \ \bar k\in\Z^m \ | \ \lambda(\bar k)\in S \ \}$ is a finite set. By this representation $u_0(x)=v_0(y(x))$, where
$$
v_0(y)=\sum_{\bar k\in J} a_{\bar k}e^{2\pi i\bar k\cdot y}
$$
is a periodic function on $\R^m$ with the standard lattice of periods $\Z^m$ while $y=\Lambda x$ is a linear map from  $\R^n$ to $\R^m$ defined by the equalities $\displaystyle y_j=\lambda_j\cdot x=\sum_{i=1}^n\lambda_{ji}x_i$,  $\lambda_{ji}$, $i=1,\ldots,n$, being coordinates of the vectors $\lambda_j$, $j=1,\ldots,m$. We consider the Hamilton-Jacobi equation
\begin{equation}\label{1r}
v_t+H(\Lambda^*\nabla_y v)=0, \quad v=v(t,y), \ t>0, \ y\in\R^m,
\end{equation}
where $\displaystyle (\Lambda^*\nabla_y v)_i=\sum_{j=1}^m \lambda_{ji}\partial_{y_j} v$, $i=1,\ldots,n$.
Let $v(t,y)$ be a v.s. of the Cauchy problem for equation (\ref{1r}) with initial function $v_0(y)$. By the periodicity of $v_0$ the function $v(t,y+e)$ is a v.s. of the same problem
for each vector $e\in\Z^m$. In view of uniqueness of v.s. we conclude that $v(t,y+e)\equiv v(t,y)$ $\forall e\in\Z^m$, i.e. $v(t,y)$ is a periodic function with respect to $y$
(with the lattice of periods $\Z^m$). In view of periodicity $v(t,\cdot)\in C([0,+\infty),C(\T^m))$, where $\T^m=\R^m/\Z^m$ is a torus
(it can be identified with the periodicity cell $[0,1)^m$~).
Let us demonstrate that $u(t,x)=v(t,\Lambda x)$.
We introduce the invertible linear operator $\tilde\Lambda$ on the extended space $\R^{n+m}$, defined by the equality $\tilde\Lambda (x,z)=(x,z+\Lambda x)$.
Since $\tilde\Lambda^*(x,y)=(x+\Lambda^*y,y)$, equation (\ref{1r}) can be rewritten in the form $v_t+H(\tilde\Lambda^*(0,\nabla_y v))=0$, where $H(p,q)=H(p)$, $p\in\R^n$, $q\in\R^m$. By Lemma~\ref{lem1} the function $v(t,x,y)=v(t,y)$ is a v.s. of equation
$v_t+H(\tilde\Lambda^*(\nabla_x v,\nabla_y v))=0$ in the extended domain $(t,x,y)\in\R_+\times\R^n\times\R^m$.
Then, by Lemma~\ref{lem2} the function $u(t,x,z)=v(t,z+\Lambda x)$ is a v.s. of (\ref{1}) considered in the extended domain $(t,x,z)\in\R_+\times\R^n\times\R^m$.  Applying Lemma~\ref{lem3},
we conclude that $u^z(t,x)=u(t,x,z)$ is a v.s. of (\ref{1}) for all $z\in\R^m$. Taking $z=0$, we find that
$u(t,x)=v(t,\Lambda x)$ is a v.s. of (\ref{1}). It is clear that $u(0,x)=v_0(\Lambda x)=u_0(x)$, that is, $u(t,x)$
is a v.s. of original problem (\ref{1}), (\ref{2}). For a fixed $t>0$ a continuous periodic function $v(t,y)$ can be uniformly approximated by finite sums $\displaystyle v_m(y)=\sum_{|\bar k|\le m} a_{m,\bar k}e^{2\pi i\bar k\cdot y}$ (for instance, one can take Fej\'er sums), so that $v_m\rightrightarrows v(t,y)$ as $m\to\infty$. This implies that $v_m(\Lambda x)\rightrightarrows u(t,x)$ as $m\to\infty$ and since
$$
v_m(\Lambda x)=\sum_{|\bar k|\le m} a_{m,\bar k}e^{2\pi i\lambda(\bar k)\cdot x}
$$
are trigonometric polynomials while $\displaystyle \lambda(\bar k)=\sum_{j=1}^m k_j\lambda_j\in M_0$,
we find that the limit function $u(t,\cdot)\in AP(\R^n)$ and $Sp(u(t,\cdot))\subset M_0$. This completes the proof in the case when $u_0(x)$ is a trigonometric polynomial.

In the general case $u_0\in AP(\R^n)$ there exists a sequence of trigonometric polynomials $u_{0m}(x)$, $m\in\N$, such that $Sp(u_{0m})\subset M_0$ and $u_{0m}\rightrightarrows u_0$ as $m\to\infty$. We can choose $u_{0m}$ as the
sequence of Bochner-Fej\'er trigonometric polynomials, see \cite{LevZh}. As it has been already established, a unique v.s. $u_m(t,x)$ of (\ref{1}), (\ref{2}) with initial data $u_{0m}$ belongs to the space $C([0,+\infty),AP(\R^n))$ and $Sp(u_m(t,\cdot))\subset M_0$ for all $t\ge 0$. Since  $\|u_m-u\|_\infty\le\|u_{0m}-u_0\|_\infty\to 0$ as $m\to\infty$,
then $u_m\rightrightarrows u$ as $m\to\infty$ and, therefore, the limit v.s. $u(t,x)\in C([0,+\infty),AP(\R^n))$ and $Sp(u(t,\cdot))\subset M_0$ for all $t\ge 0$.
The proof is complete.
\end{proof}

\begin{remark}
In view of Corollary~\ref{cor1} and translation invariance of v.s. for all $t>0$
$$
\|u(t,x+l)-u(t,x)\|_\infty\le\|u_0(x+l)-u_0(x)\|_\infty.
$$
This estimate implies that any $\varepsilon$-almost-period $l$ of $u_0(x)\in AP(\R^n)$ is a common $\varepsilon$-almost period of $u(t,\cdot)$ for all $t\ge 0$. From this and the known relation between the $\varepsilon$-almost-periods and the spectrum  it readily follows the assertions of Theorem~\ref{th2}. This reasons seem to be more simple but we preferred the proof which does not use the notion of almost-periods and based on reduction to the periodic case because this reduction will be utilized also in the proof of Theorem~\ref{thM1} below.
\end{remark}

\section{The case of single space variable}

\subsection{Long time behavior of periodic v.s.}

Now we consider the case of single space variable when our equation (\ref{1}) has the form
\begin{equation}\label{1d1}
u_t+H(u_x)=0,
\end{equation}
$u=u(t,x)$, $(t,x)\in\Pi=\R_+\times\R$.
We are going to investigate the long time behavior of almost periodic viscosity solutions of the Cauchy problem for equation (\ref{1d1})
with initial condition
\begin{equation}\label{1d2}
u(0,x)=u_0(x).
\end{equation}
In the one-dimensional case there is a direct connection between v.s. of (\ref{1d1}), (\ref{1d2}) and entropy solutions (in Kruzhkov sense \cite{Kr}~) of the corresponding Cauchy problem for the conservation law
\begin{equation}\label{cl1}
v_t+H(v)_x=0
\end{equation}
with initial data
\begin{equation}\label{clI}
v(0,x)=v_0(x).
\end{equation}
More precisely, assume that $u_0(x)$ is Lipshitz continuous, i.e. its generalized derivatives $u_0'(x)\in L^\infty(\R)$. Then the unique v.s. $u(t,x)$ of (\ref{1d1}), (\ref{1d2}) also  satisfies the Lipschitz condition with respect to the space variable and the derivative $v=u_x(t,x)\in L^\infty(\Pi)$ is a unique entropy solution of (\ref{cl1}), (\ref{clI}) with $v_0(x)=u_0'(x)$, cf. \cite{CL,KrHJ1}. Observe that (\ref{cl1}) can be derived from (\ref{1d1}) by formal differentiation: $v_t+H(v)_x=u_{xt}+H(u_x)_x=0$, $v=u_x$.

Recall the notion of entropy solution (e.s. for short) of the Cauchy problem for a multidimensional conservation law
\begin{equation}\label{cons}
v_t+\div_x f(v)=0,
\end{equation}
$v=v(t,x)$, $(t,x)\in\Pi=\R_+\times\R^n$, with merely continuous flux vector $f(v)=(f_1(v),\ldots,f_n(v))\in C(\R,\R^n)$ and with initial condition
\begin{equation}\label{consI}
v(0,x)=v_0(x)\in L^\infty(\R^n).
\end{equation}

\begin{definition}\label{def2} A bounded measurable function $v=v(t,x)\in L^\infty(\Pi)$ is called an e.s. of (\ref{cons}), (\ref{consI}) if for all $k\in\R$
\begin{equation}
\label{en1}
\frac{\partial}{\partial t} |v-k|+\div_x [\sign(v-k)(f(v)-f(k))]\le 0
\end{equation}
in the sense of distributions on $\Pi$ (in $\D'(\Pi)$), and
$$
\esslim_{t\to 0+} v(t,\cdot)=v_0 \quad \mbox{ in } L^1_{loc}(\R^n).
$$
\end{definition}

Here $\sign u=\left\{\begin{array}{rr} 1, & u>0, \\ -1, & u\le 0
\end{array}\right.$
and relation (\ref{en1}) means that for each test function $h=h(t,x)\in C_0^1(\Pi)$,
$h\ge 0$,
$$
\int_\Pi \,[|v-k|h_t+\sign(v-k)(f(v)-f(k))\cdot\nabla_x h]dtdx\ge 0.
$$
Taking in (\ref{en1}) $k=\pm R$, where $R\ge\|v\|_\infty$, we obtain that $v_t+\div_x f(v)=0$ in $\D'(\Pi)$, that is, an e.s. $v=v(t,x)$ is a weak solutions of this equation as well.
It was also established in \cite[Corollary~7.1]{Pa6} that, after possible correction on a set of null
measure, an e.s. $u(t,x)$ is continuous on $[0,+\infty)$ as a map $t\mapsto u(t,\cdot)$ into
$L^1_{loc}(\R^n)$. In the sequel we will always assume that this property is satisfied.

Suppose that the initial function $v_0$ is periodic with a lattice of periods $\LL$, i.e., $v_0(x+e)=v_0(x)$ a.e. on $\R^n$ for every $e\in \LL$. Let $\T^n=\R^n/\LL$ be the corresponding $n$-dimensional torus, and $\LL'$ be the dual lattice $$\LL'=\{ \ \xi\in\R^n \ | \ \xi\cdot x\in\Z \ \forall x\in \LL \ \}.$$
In the case under consideration when the flux vector is merely continuous the property of finite speed of propagation for initial perturbation may be violated, which, in the multidimensional situation $n>1$, may even lead to the nonuniqueness of e.s. to Cauchy problem (\ref{cons}), (\ref{consI}), see examples in \cite{KrPa1,KrPa2}. But for a periodic initial function $v_0(x)$, an e.s. $v(t,x)$ of (\ref{cons}), (\ref{consI}) is unique (in the class of all e.s., not necessarily periodic) and space-periodic, the proof can be found in \cite{PaMax1}. It is also shown in \cite{PaMax1} that the mean value of e.s. over the period does not depend on time:
\begin{equation}\label{mass}
\int_{\T^n}v(t,x)dx=I\doteq\int_{\T^n}v_0(x)dx,
\end{equation}
where $dx$ is the normalized Lebesgue measure on $\T^n$.

The following theorem, proven in \cite{PaNHM}, generalizes the previous results of \cite{Daf,PaAIHP}.

\begin{theorem}\label{th3}
Suppose that
\begin{eqnarray}\label{ND2}
\forall\xi\in \LL', \xi\not=0 \ \mbox{ the function } u\to\xi\cdot f(v) \nonumber\\ \mbox{ is not affine on any vicinity of } I.
\end{eqnarray}
Then
\begin{equation}\label{dec}
\lim_{t\to +\infty} v(t,\cdot)=I \ \mbox{ in } L^1(\T^n).
\end{equation}
\end{theorem}

Actually, condition (\ref{ND2}) is necessary and sufficient for the decay property (\ref{dec}).
Now we return to the case of one space variable and assume that initial function $v_0\in L^\infty(\R)$ is periodic: $v_0(x+1)=v_0(x)$ a.e. in $\R$. This can be written as $v_0\in L^\infty(\T)$, where $\T=\R/\Z$ is a circle (~it can be identified with the periodicity cell $[0,1)$~). Denote by
$I$ the mean value of $v_0$: $I=\int_\T v_0(x) dx=\int_0^1v_0(x) dx$. The unique e.s. $v=v(t,x)$ of (\ref{cl1}), (\ref{clI}) is also $x$-periodic, $v\in C([0,+\infty),L^1(\T))$, and, as was shown in recent paper \cite{PaMZM}, the following long time asymptotic property holds.

\begin{theorem}\label{th4}
  There exist a periodic function $w(y)\in L^\infty(\T)$ (profile) and a constant $c\in\R$ (speed) such that
  \begin{equation}\label{5}
  v(t,x)-w(x-ct)\mathop{\to}_{t\to+\infty} 0 \ \mbox{ in } L^1(\T).
  \end{equation}
 Moreover, $\int_\T w(y)dy=I$, and the function $H(v)-cv\equiv\const$ on the interval $(\essinf w(y),\esssup w(y))$.
 \end{theorem}
Notice that in the case when the function $H(v)$ is not affine in any vicinity of $I$ it follows from Theorem~\ref{th4} that $w(y)\equiv I$ and (\ref{5}) reduces to the decay property
$$
v(t,x)\mathop{\to}_{t\to+\infty} I \ \mbox{ in } L^1(\T)
$$
from the previous Theorem~\ref{th3} (in the case $n=1$).

\medskip
Making the change $\tilde u=u+H(0)t$, which transforms the equation (\ref{1d1}) to the equation $\tilde u_t+H(\tilde u_x)-H(0)=0$,
we may suppose that $H(0)=0$. Then constants are v.s. of (\ref{1d1}), and by Corollary~\ref{cor1} with $u_1=u$, $u_2=0$ we find that a v.s. $u=u(t,x)$ is bounded, namely $\|u\|_\infty\le\|u_0\|_\infty$.

Let $(a,b)$, $-\infty\le a<0<b<+\infty$ be the maximal neighborhood of zero such that $H(u)$ is linear on $(a,b)$: $H(u)=cu$ on $(a,b)$ for some $c\in\R$. If such an interval does not exist, i.e. $H(u)$ is not linear in any vicinity of zero, we set $a=b=0$. Assume that the initial function $u_0(x)$ is periodic, $u_0(x)\in C(\T)$, and $u(t,x)$ is a unique v.s. of (\ref{1d1}), (\ref{1d2}). Since $u(t,x+1)$ is a v.s. of the same problem, then $u(t,x+1)=u(t,x)$ by the uniqueness of v.s. The next our result is similar to Theorem~\ref{th4}, this is about the long time convergence of $u(t,x)$ to a traveling wave.

\begin{theorem}\label{thM}
There exist a continuous periodic function $p(y)\in C(\T)$ and a real constant $c$ such that
\begin{equation}\label{6}
u(t,x)-p(x-ct)\rightrightarrows 0 \quad \mbox{ as } t\to+\infty.
\end{equation}
Moreover, the profile $p(y)$ satisfies the one-sided Lipschitz estimates
\begin{equation}\label{7}
a(y_2-y_1)\le p(y_2)-p(y_1)\le b(y_2-y_1) \quad \forall y_1,y_2\in\R, y_2>y_1,
\end{equation}
while the speed $c$ is determined by the condition that $H(u)=cu$ on $(a,b)$. In the case $a=b=0$
it follows from (\ref{7}) that $p(y)\equiv p_0=\const$ and (\ref{6}) reduces to the following decay property
\begin{equation}\label{8}
u(t,x)\rightrightarrows p_0 \quad \mbox{ as } t\to+\infty
\end{equation}
(in this case the value of $c$ does not matter).
\end{theorem}

\begin{proof}
First, we consider the case when $u_0$ is Lipschitz. Then $v=u_x(t,x)$ is an e.s. of (\ref{cl1}), (\ref{clI}) with initial data $v_0(x)=u_0'(x)\in L^\infty(\T)$. Obviously, $I=\int_\T v_0(x)dx=0$ and by Theorem~\ref{th4}
$v(t,x)$ converges as $t\to+\infty$ to a traveling wave $w(x-ct)$ in the sense of relation (\ref{5}). In view of Theorem~\ref{th4}, we also find that $\int w(y)dy=0$ and that $a\le w(y)\le b$. Therefore, the function $\tilde p(y)=\int_0^y w(s)ds$ is periodic continuous function satisfying (\ref{7}) (with $\tilde p$ instead of $p$).
Notice that $\tilde p\equiv 0$ in the case $a=b=0$. Let $c\in\R$ is determined by the condition $H(u)=cu$ on $(a,b)$ and is chosen arbitrarily if $a=b=0$. Now it follows from (\ref{6}) that
\begin{equation}\label{9}
\frac{\partial}{\partial x}(u(t,x)-\tilde p(x-ct))=v(t,x)-w(x-ct)\mathop{\to}_{t\to+\infty} 0 \ \mbox{ in } L^1(\T). \end{equation}
Since $H(u)\equiv cu$ on the range of $\tilde p_x$, the traveling wave $\tilde p(x-ct)$ is a v.s. of equation
(\ref{1d1}), $\tilde p_t+H(\tilde p_x)=\tilde p_t+c\tilde p_x=0$. Then, as it readily follows from Corollary~\ref{cor1}, where we replace the initial time $t=0$ by $t=\tau$,
$$
\min_{x\in\T} (u(\tau,x)-p(x-c\tau))\le u(t,x)-p(x-ct)\le \max_{x\in\T} (u(\tau,x)-p(x-c\tau))
$$
for all $x\in\T$, $t,\tau\ge 0$, $t>\tau$. Therefore, the functions
$$
m(t)=\min_{x\in\T} (u(t,x)-p(x-ct)), \quad M(t)=\max_{x\in\T} (u(t,x)-p(x-ct))
$$
are, respectively, increasing and decreasing on $[0,+\infty)$, and $m(t)\le M(t)$. This implies existence of limits
$$
m_\infty=\lim_{t\to+\infty} m(t), \quad M_\infty=\lim_{t\to+\infty} M(t).
$$
In view of (\ref{9}) we also find that
$$
M(t)-m(t)\le \int_\T\left|\frac{\partial}{\partial x}(u(t,x)-\tilde p(x-ct))\right|dx\to 0
$$
as $t\to+\infty$. Hence $m_\infty=M_\infty=m_*$. It now follows from the estimate $m(t)\le u(t,x)-\tilde p(x-ct)\le M(t)$ and the limit relation
$$
\lim_{t\to+\infty}m(t)=\lim_{t\to+\infty}M(t)=m_*,
$$
that $u(t,x)-\tilde p(x-ct)\rightrightarrows m_*$ as $t\to+\infty$. Taking $p(y)=m_*+\tilde p(y)$, we conclude that
(\ref{6}) holds. Clearly, condition (\ref{7}) is also satisfied.

In the general case $u_0\in C(\T)$ we consider a sequence $u_{0n}\in C(\T)$, $n\in\N$ of Lipschitz functions, such that
$u_{0n}\to u_0$ as $n\to\infty$ in $C(\T)$. Let $u_n=u_n(t,x)$ be a v.s. of (\ref{1d1}), (\ref{1d2}) with initial data $u_{0n}$. Then, by Corollary~\ref{cor1},
\begin{equation}\label{10}
\|u_n(t,\cdot)-u(t,\cdot)\|_\infty\le\|u_{0n}-u_0\|_\infty\mathop{\to}_{n\to\infty} 0.
\end{equation}
As we have already established, there exist functions $p_n(y)\in C(\T)$ such that
\begin{eqnarray}\label{11}
u_n(t,x)-p_n(x-ct)\rightrightarrows 0 \quad \mbox{ as } t\to+\infty, \\
\label{12}
a(y_2-y_1)\le p_n(y_2)-p_n(y_1)\le b(y_2-y_1) \quad \forall y_1,y_2\in\R, y_2>y_1.
\end{eqnarray}
Notice that the speed $c$ does not depend on $n$. It follows from (\ref{11}) and (\ref{10}) that for all $m,n\in\N$
$$
\|p_n-p_m\|_\infty\le\|u_n-u_m\|_\infty\to 0 \ \mbox{ as } n,m\to\infty,
$$
that is, $p_n(y)$ is a Cauchy sequence in $C(\T)$. Since this space is complete, there exists a limit function $p(y)\in C(\T)$, so that $p_n(y)\rightrightarrows p(y)$ as $n\to\infty$. Now (\ref{6}), (\ref{7}) follows from (\ref{11}), (\ref{12}) in the limit as $n\to\infty$. The proof is complete.
\end{proof}

\begin{remark}
As is easy to see, in the case when $H(0)$ may be arbitrary, relation (\ref{6}) should be replaced by the following one
$$
u(t,x)+H(0)t-p(x-ct)\rightrightarrows 0 \quad \mbox{ as } t\to+\infty.
$$
\end{remark}

\subsection{Decay of almost periodic v.s.}

Now we assume that initial function in (\ref{1d2}) is almost periodic: $u_0\in AP(\R)$. Suppose also that $H(0)=0$. By Theorem~\ref{th2} a unique v.s. $u(t,x)$ of (\ref{1d1}), (\ref{1d2}) is also almost periodic over the space variables: $u(t,x)\in C([0,+\infty),AP(\R))$, and $Sp(u(t,\cdot)\subset M_0$, where $M_0$ is an additive subgroup of $\R$ spanned by $Sp(u_0)$.
It turns out that this v.s. satisfies the same decay property as in periodic case, cf. Theorem~\ref{thM}.

\begin{theorem}\label{thM1}
Assume that $H(u)$ is not linear in any vicinity of zero. Then
$$
u(t,\cdot)\rightrightarrows c=\const \quad \mbox{ as } t\to+\infty.
$$
\end{theorem}
\begin{proof}
It seems natural to reduce the Hamilton-Jacobi equation to conservation law (\ref{cl1}), like in the proof of Theorem~\ref{thM} above, and to use the decay properties of almost periodic e.s. Unfortunately, the known decay property of almost periodic e.s. of conservation laws (see, for example, \cite{PaJHDE1}) asserts the decay in the Besicovitch norm, which is too weak for our aims. We will follow another approach similar to one used in the proof of Theorem~\ref{th2}, which is based on reduction to the periodic case.

Let us first consider the case when $\displaystyle u_0(x)=\sum_{\lambda\in S}a_\lambda e^{2\pi i\lambda x}$ is a trigonometric polynomial. In this case $S=Sp(u_0)$ is finite and the group $M_0\subset\R$ generated by the spectrum $S$ is a free abelian group of finite rank. Let $\lambda_j\in\R$, $j=1,\ldots,m$ be a basis of $M_0$, $\Lambda=(\lambda_1,\ldots,\lambda_m)\in\R^m$. Then, see the proof of Theorem~\ref{th2}, $u_0(x)=v_0(y(x))$, $u(t,x)=v(t,y(x))$, where $\displaystyle v_0(y)=\sum_{\bar k\in J} a_{\bar k}e^{2\pi i\bar k\cdot y}\in C(\T^m)$ is a periodic trigonometric polynomial on $\R^m$ with the standard lattice of periods $\Z^m$, $J=\{ \bar k=(k_1,\ldots,k_m)\in\Z^m \ | \lambda(\bar k)=\Lambda\cdot\bar k\in S \ \}$, $y(x)=x\Lambda$ is a linear map from $\R$ into $\R^m$, and $v(t,y)\in C([0,+\infty),C(\T^m))$ is a unique v.s. of the problem
\begin{equation}\label{red1}
v_t+H(\Lambda\cdot\nabla_y v)=0, \quad v(0,y)=v_0(y).
\end{equation}

Since $v_0(y)$ is a trigonometric polynomial, it satisfies the Lipschitz condition $|v_0(y+z)-v_0(z)|\le L|z|$ $\forall y,z\in\R^m$, where $L>0$ is a Lipschitz constant. Obviously,
$v(t,y+z)$ is a v.s. of (\ref{red1}) with initial function $v_0(y)$ replaced by $v_0(y+z)$. By Corollary~\ref{cor1}
$|v(t,y+z)-v(t,y)|\le\|v_0(\cdot+z)-v_0\|_\infty\le L|z|$ $\forall {t\ge 0},y,z\in\R^m$, that is, $v(t,y)$ is Lipschitz continuous with respect to the space variables $y$. In particular, $\nabla_y v(t,y)\in L^\infty(\Pi,\R^m)$, and $\|\nabla_y v(t,y)\|_\infty\le L$. By our assumptions $H(0)=0$, which guarantees the estimate $\|v\|_\infty\le\|v_0\|_\infty$.
We see that the family $v(t,\cdot)$, $t>0$ is bounded and equicontinuous in $C(\T^m)$. By the Arzel\`a-Ascoli theorem there exist a sequence $t_r>0$, $r\in\N$ such that
$t_r\to+\infty$ as $r\to\infty$ and a function $v_\infty=v_\infty(y)\in C(\T^m)$ with the property $v(t_r,\cdot)\rightrightarrows v_\infty$ as $r\to\infty$.

Let us apply the directional derivative $D_\Lambda=\Lambda\cdot\nabla$ to equation (\ref{red1}). We obtain
$$
(D_\Lambda v)_t+ D_\Lambda H(D_\Lambda v)=0.
$$
After the change $w=D_\Lambda v$ we arrive at the conservation law
\begin{equation}\label{red2}
w_t+\div_y (H(w)\Lambda) =w_t+\sum_{j=1}^m (\lambda_j H(w))_{y_j}=0
\end{equation}
with the flux vector $f(w)=H(w)\Lambda$, equipped with the corresponding initial condition
\begin{equation}\label{red3}
w(0,y)=w_0(y)\doteq D_\Lambda v_0(y)\in L^\infty(\T^m).
\end{equation}

Using the vanishing viscosity approximations like in \cite{CL}, one can claim that $w=D_\Lambda v(t,y)$ is an e.s. of (\ref{red2}), (\ref{red3}). For the sake of completeness we provide more details. As was shown in \cite{CL}, the v.s. $v(t,y)$ is an uniform limit of the sequence $v_l(t,y)\in C^2(\Pi)$, $l\in\N$, of classical solutions to the parabolic problems
\begin{equation}\label{red4}
v_t+H_l(\Lambda\cdot\nabla_y v)=\varepsilon_l\Delta v, \quad v(0,y)=v_0(y),
\end{equation}
where $0<\varepsilon_l\to 0$ as $l\to\infty$, and $H_l(s)$ is a sufficiently regular approximation of the hamiltonian $H(s)$,
Applying the derivative $D_\Lambda$ to equation (\ref{red4}), we find that the function $w_l(t,y)=D_\Lambda v$ is a classical solution to the problem
$$
w_t+\sum_{j=1}^m (\lambda_j H_l(w))_{y_j}=\varepsilon_l\Delta w, \quad w(0,y)=w_0(y).
$$
In the correspondence with vanishing viscosity method (see, for example, \cite{Kr}), the sequence $w_l(t,y)=D_\Lambda v_l\to w(t,y)$ as $l\to \infty$ in $L^1_{loc}(\Pi)$, where
$w(t,y)$ is an e.s. of (\ref{red2}), (\ref{red3})
As was noticed above, $v_l\rightrightarrows v$ as $l\to\infty$ and we conclude that $D_\lambda v=w$ in $\D'(\Pi)$.

Observe that $w_0(y)=D_\Lambda v_0(y)$ is a periodic function with zero mean value. Further, for each $\xi\in\Z^m$, $\xi\not=0$
$$
\xi\cdot H(w)\Lambda=(\xi\cdot\Lambda)H(w)
$$
and since $\lambda_j$, $j=1,\ldots,m$, is a basis of $M_0$,
$$\xi\cdot\Lambda=\sum_{j=1}^m \xi_j\lambda_j\not=0.$$
Therefore, $\xi\cdot (H(w)\Lambda)=(\xi\cdot\Lambda) H(w)$ is not linear in any vicinity of zero for all $\xi\in\Z^m$, $\xi\not=0$.
By Theorem~\ref{th3} (with $\LL'=\LL=\Z^m$)
$w(t,y)\to 0$ as $t\to\infty$ in $L^1(\T^m)$. Taking $t=t_r$, we find that
\begin{eqnarray*}
w(t_r,y)=D_\Lambda v(t_r,y)\mathop{\to}_{r\to\infty} 0 \ \mbox{ in } L^1(\T^m), \\
v(t_r,y)\mathop{\to}_{r\to\infty} v_\infty(y) \ \mbox{ in } C(\T^m).
\end{eqnarray*}
From these equalities it follows that $D_\Lambda v_\infty(y)=0$ in $\D'(\R^m)$. Taking also into account that
$v_\infty(y)$ is a continuous function we find that $v_\infty(x\Lambda)=c=\const$ $\forall x\in\R$. Since the numbers $\lambda_j$, $j=1,\ldots,m$, are linear independent over the ring $\Z$, then by the Kronecker's theorem (~see \cite{LevZh}~) the curve $y=x\Lambda$ is dense in torus $\T^m$. This, together with continuity of $v_\infty$, implies that $v_\infty\equiv c$. By Corollary~\ref{cor1} we find
$$\|v(t,\cdot)-c\|_\infty\le \|v(t_l,\cdot)-c\|_\infty \quad \forall t>t_l.$$
From this it follows that
$$
\lim_{t\to+\infty}\|v(t,\cdot)-c\|_\infty=\lim_{l\to\infty}\|v(t_l,\cdot)-c\|_\infty=0.
$$
To complete the proof, it only remain to notice that $u(t,x)=v(t,x\Lambda)$.

In general case of arbitrary almost periodic initial function $u_0(x)$
we construct a sequence $u_{0n}$, $n\in\N$, of trigonometric polynomials such that
$u_{0n}\to u_0$ as $n\to\infty$ in $AP(\R)$. Let $u_n=u_n(t,x)$ be a v.s. of (\ref{1d1}), (\ref{1d2}) with initial data $u_{0n}$. Then, as follows from Corollary~\ref{cor1},
\begin{equation}\label{red5}
\|u_n(t,\cdot)-u(t,\cdot)\|_\infty\le\|u_{0n}-u_0\|_\infty\mathop{\to}_{n\to\infty} 0.
\end{equation}
As we have already established, there exist constants $c_n$ such that
\begin{equation}\label{red6}
u_n(t,x)\rightrightarrows c_n \quad \mbox{ as } t\to+\infty.
\end{equation}
It follows from (\ref{red5}) and (\ref{red6}) that for all $m,n\in\N$
$$
|c_n-c_m|\le\|u_n-u_m\|_\infty\to 0 \ \mbox{ as } n,m\to\infty,
$$
that is, $c_n$ is a Cauchy sequence in $\R$. Therefore, $c_n\to c$ as $n\to\infty$, where $c$ is some constant. It now follows from   (\ref{red5}), (\ref{red6}) in the limit as $n\to\infty$ that $u(t,\cdot)\rightrightarrows c$ as $t\to+\infty$. The proof is complete.
\end{proof}

\textbf{Acknowledgments}.
The research was carried out under support of the Russian Foundation for Basic Research (grant no. 15-01-07650-a)  and the Ministry of Education and Science of Russian Federation (project no. 1.445.2016/1.4).


\begin{thebibliography}{99}
\bibitem{CL}
Crandall M.G., Lions P.L. Viscosity solutions of Hamilton-Jacobi equations. Trans. Amer. Math. Soc. 1983. Vol. 277, pp. 1--42.
\bibitem{CEL}
Crandall M.G., Evans L.C., Lions P.L. Some properties of viscosity solutions of Hamilton-Jacobi Equations. Trans. Amer. Math. Soc. 1984. Vol. 282(2), pp. 487--502.
\bibitem{Daf}
Dafermos C.M. Long time behavior of periodic solutions to scalar conservation laws in several space dimensions.
SIAM J. Math. Anal. 2013. Vol. 45, pp. 2064--2070.
\bibitem{KrHJ}
Kruzhkov S.N. Generalized solutions of nonlinear first order equations with several independent variables, I. Mat. Sb. 1966. Vol. 70(3), pp. 394--415.
\bibitem{KrHJ1}
Kruzhkov S.N. Generalized solutions of nonlinear first order equations with several independent variables, II.
Math. USSR-Sb. 1967. Vol. 1(1), pp. 93--116.
\bibitem{Kr}
Kruzhkov S.N. First order quasilinear equations in several independent variables. Math. USSR Sb. 1970. Vol. 10, pp. 217--243.
\bibitem{KrPa1}
Kruzhkov S.N., Panov E.Yu.
First-order conservative quasilinear laws with an infinite domain of dependence on the initial data.
Soviet Math. Dokl. 1991. Vol. 42, pp. 316--321.
\bibitem{KrPa2}
Kruzhkov S.N., Panov E.Yu. Osgood's type conditions for uniqueness of entropy solutions to Cauchy problem for quasilinear conservation laws of the first order.
Ann. Univ. Ferrara Sez. VII (N.S.). 1994. Vol. 40, pp. 31--54.
\bibitem{Lang}
Lang S. Algebra (Revised 3rd ed.). New York: Springer-Verlag, 2002.
\bibitem{LevZh}
Levitan B.M., Zhikov V.V. Almost periodic functions and differential equations. Cambridge University Press, 1982.
\bibitem{PaMax1}
Panov E.Yu.
A remark on the theory of generalized entropy sub- and supersolutions of the Cauchy problem for a first-order quasilinear equation.
Differ. Equ. 2001. Vol. 37, pp. 272--280.
\bibitem{Pa6}
Panov E.Yu. Existence of strong traces for generalized solutions of multidimensional scalar conservation laws.
J. Hyperbolic Differ. Equ. 2005. Vol. 2, pp. 885--908.
\bibitem{PaAIHP}
Panov E.Yu.
On decay of periodic entropy solutions to a scalar conservation law.
Ann. Inst. H. Poincar\'e Anal. Non Lin\'eaire. 2013. Vol. 30, pp. 997--1007.
\bibitem{PaNHM}
Panov E.Yu. On a condition of strong precompactness and the decay of periodic entropy solutions to scalar conservation laws. Netw. Heterog. Media. 2016. Vol.~11(2), pp.~349--367.
\bibitem{PaMZM}
Panov E.Yu. Long time asymptotics of periodic generalized entropy solutions of scalar conservation laws. Math. Notes. 2016. Vol. 100, pp. 112--121.
\bibitem{PaJHDE1}
Panov E.Yu. On the Cauchy problem for scalar conservation laws in the class of Besicovitch almost periodic functions: Global well-posedness and decay property. J. Hyperbolic Differ. Equ. 2016. Vol. 13, pp. 633--659.
\end{thebibliography}
\end{document}